\documentclass{conm-p-l}

\usepackage{euscript}
\usepackage{amsmath}
\usepackage{amsthm}
\usepackage{epsfig}
\usepackage{amssymb}

\DeclareMathOperator{\Hom}{{\rm Hom}}

\DeclareMathOperator{\coker}{{\rm coker}}

\DeclareMathOperator{\rank}{{\rm rank}}

\numberwithin{equation}{section}
\numberwithin{equation}{subsection}

\theoremstyle{plain}
\newtheorem{theorem}[equation]{Theorem}
\newtheorem{lemma}[equation]{Lemma}
\newtheorem{proposition}[equation]{Proposition}
\newtheorem{corollary}[equation]{Corollary}

\theoremstyle{definition}
\newtheorem{example}[equation]{Example}
\newtheorem{remark}[equation]{Remark}

\newtheorem{bekezdes}[equation]{}

\numberwithin{equation}{section}
\numberwithin{equation}{subsection}

\title{Two exact sequences for lattice cohomology}

\author{Andr\'as N\'emethi}
\address{A. R\'enyi Institute of Mathematics, 1053 Budapest,  Re\'altanoda u. 13-15,  Hungary.}
\email{nemethi@renyi.hu}
\thanks{The author is partially supported by OTKA Grant K67928.}



\date{}

\begin{document}

\maketitle

\begin{center}
 {\em Dedicated to Henri Moscovici on his 60th birthday.}
\end{center}

\pagestyle{myheadings} \markboth{{\normalsize Andr\'as
N\'emethi}}{ {\normalsize Two exact sequences for lattice cohomology}}

\newcommand{\ssw}{{\bf sw}}
\newcommand{\et}{\mathcal{T}}
\newcommand{\bS}{{\mathbb S}}
\newcommand{\bma}{\mbox{\boldmath$a$}}
\newcommand{\bmb}{\mbox{\boldmath$b$}}
\newcommand{\bmc}{\mbox{\boldmath$c$}}
\newcommand{\bme}{\mbox{\boldmath$e$}}
\newcommand{\bmi}{\mbox{\boldmath$i$}}
\newcommand{\bmj}{\mbox{\boldmath$j$}}
\newcommand{\bmv}{\mbox{\boldmath$v$}}
\newcommand{\bmk}{\mbox{\boldmath$k$}}
\newcommand{\bmm}{\mbox{\boldmath$m$}}
\newcommand{\bms}{\mbox{\boldmath$s$}}
\newcommand{\bSW}{\mbox{\boldmath$SW$}}
\newcommand{\bmf}{\mbox{\boldmath$f$}}
\newcommand{\bmg}{\mbox{\boldmath$g$}}
\newcommand{\gq}{{\mathfrak q}}
\newcommand{\xo}{o}
\newcommand{\veeK}{{\vee}}
\newcommand{\gh}{g}
\newcommand{\cH}{\mathcal{H}}

\newcommand{\lp}{{l}}
\newcommand{\ev}{\varepsilon}
\newcommand{\tx}{\tilde{X}}
\newcommand{\tz}{\tilde{Z}}
\newcommand{\call}{{\mathcal L}}
\newcommand{\calm}{{\mathcal M}}
\newcommand{\calx}{{\mathcal X}}
\newcommand{\calo}{{\mathcal O}}
\newcommand{\calt}{{\mathcal T}}
\newcommand{\cali}{{\mathcal I}}
\newcommand{\calj}{{\mathcal J}}
\newcommand{\calC}{{\mathcal C}}
\newcommand{\calS}{{\mathcal S}}
\newcommand{\calQ}{{\mathcal Q}}
\newcommand{\calF}{{\mathcal F}}

\newcommand{\cs}{\langle \chi_0\rangle}

\newcommand{\cc}{\bar{C}}
\newcommand{\vp}{\varphi}

\def\mmod{\mbox{mod}}
\let\d\partial
\def\EE{\mathcal E}
\newcommand{\cC}{\EuScript{C}}
\def\C{\mathbb C}
\def\Q{\mathbb Q}
\def\R{\mathbb R}
\def\bS{\mathbb S}
\def\bH{\mathbb H}
\def\bB{\mathbb B}\def\bC{\mathbb C}\def\bA{\mathbb A}
\def\Z{\mathbb Z}
\def\N{\mathbb N}
\def\bn{\mathbb N}
\def\bp{\mathbb P}\def\bt{\mathbb T}
\def\eop{$\hfill\square$}
\def\bif{(\, , \,)}
\def\coker{\mbox{coker}}
\def\im{{\rm Im}}

\newcommand{\Gammma}{{G}}
\newcommand{\no}{\noindent}
\newcommand{\bfc}{{\mathbb C}}
\newcommand{\bfq}{{\mathbb Q}}
\newcommand{\cale}{{\mathcal E}}
\newcommand{\calw}{{\mathcal W}}
\newcommand{\calv}{{\mathcal V}}
\newcommand{\calP}{{\mathcal P}}
\newcommand{\calI}{{\mathcal I}}
\newcommand{\cala}{{\mathcal A}}
\newcommand{\calr}{{\mathcal R}}
\newcommand{\calG}{{\mathcal G}}
\newcommand{\bc}{{\mathbb C}}
\newcommand{\bez}{B_{\epsilon_0}}
\newcommand{\br}{{\mathbb R}}
\newcommand{\bq}{{\mathbb Q}}
\newcommand{\sez}{S_{\epsilon_0}}
\newcommand{\ep}{\epsilon}
\newcommand{\vs}{\vspace{3mm}}
\newcommand{\si}{\sigma}
\newcommand{\Gammas}{S}
\newcommand{\Gj}{\Gamma\setminus j_0}
\newcommand{\Gjp}{\Gamma_{ j_0}^+}
\newcommand{\q}{w}
\newcommand{\btt}{{\bf t}}

\begin{abstract}
This article is a continuation of \cite{NLC}, where the lattice cohomology of
connected negative definite plumbing
graphs was introduced. Here we consider the more general situation of non--degenerate plumbing graphs, and
we establish two exact sequences for their lattice cohomology.
The first is the analogue of the surgery
exact triangle proved by Ozsv\'ath and Szab\'o for the Heegaard--Floer invariant $HF^+$; for the lattice cohomology
over $\Z_2$--coefficients it was proved  by J. Greene in \cite{JG}. Here we prove it over $\Z$,
and we supplement it by some additional properties valid for negative definite graphs.
The second exact sequence is an adapted version which does not mix the
classes of the characteristic elements ($spin^c$--structures); it was partially motivated by the surgery
formula for the Seiberg--Witten invariant obtained  in  \cite{BN}.
For this, we define the {\it relative lattice cohomology} and we determine its
Euler characteristic in terms of Seiberg--Witten invariants.
\end{abstract}


{\small

\section{Introduction}
The lattice cohomology $\{\bH^q(\Gamma)\}_{q\geq 0}$  was introduced in \cite{NLC}. In its original version,
it was associated with any connected negative definite plumbing graph $\Gamma$, or, equivalently, with
any oriented 3--manifold, which might appear as the link of a local complex normal surface singularity.
Lattice cohomology (together with the graded roots) plays a crucial role   in the comparison of the analytic and topological invariants of surface singularities, cf. \cite{NOSZ,trieste,NLC},
see also \cite{BN,NN1,NSW} for relations with the Seiberg--Witten invariants of the link.  Additionally, the lattice cohomology (conjecturally)
offers a combinatorial description for the Heegaard--Floer homology of Ozsv\'ath and Szab\'o (for
this theory see \cite{OSzP,OSz,OSzTr} and the long list of articles
of Ozsv\'ath and Szab\'o). Indeed,  in \cite{NLC} the author conjectured that
$\oplus_{q\ even/odd}\bH^q(\Gamma)$ is isomorphic as a graded $\Z[U]$--module with
$HF^+_{even/odd}(-M(\Gamma))$, where $M(\Gamma)$ is the plumbed 3--manifold associated with $\Gamma$. (Recall that at this moment there is no combinatorial definition/characterization of $HF^+$.)

For rational and `almost rational' graphs this correspondence was established in \cite{OSzP,NOSZ}
(see also \cite{NR} for a different situation  and \cite{NLC} for related results).
A possible machinery which might help  to prove the general conjecture
is based on the surgery exact sequences. They are established for the Heegaard--Floer theory in the work
of Ozsv\'ath and Szab\'o. Our goal is to prove  the analogous exact sequences for the lattice cohomology.
In fact, independently of the above conjecture and correspondence, the proof of such exact sequences
is of major importance, and they are fundamental in the computation and in finding the main properties of the lattice cohomologies.

The formal, combinatorial definition of the lattice cohomology permits to extend its definition to
arbitrary graphs (plumbed 3--manifolds), the connectedness and negative definiteness
assumptions can be dropped.
Nevertheless, in the proof of the exact sequences,
we will deal only with non--degenerate graphs (they are those
graphs whose associated intersection form has non--zero determinant).

More precisely, for any graph $\Gamma$ and fixed vertex $j_0$, we consider the graphs
$\Gamma\setminus j_0$ and
$\Gjp$, where the first one is obtained from $\Gamma$ by deleting the vertex $j_0$ and adjacent edges,
while the second one is obtained from $\Gamma$ by replacing the decoration $e_{j_0}$
of the vertex $j_0$  by $e_{j_0}+1$.
We will assume that all these graphs are non--degenerate. Then Theorem \ref{ES}
establishes the following long exact sequence.

\vspace{2mm}

\noindent {\bf Theorem A.} {\it  Assume that the graphs $\Gjp, \ \Gamma$ and $\Gj$ are non--degenerate.
Then
$$\cdots \longrightarrow \bH^q(\Gjp)\stackrel{\bA^q}{\longrightarrow} \bH^q(\Gamma)\stackrel{\overline{\bB}^q}{\longrightarrow} \bH^q(\Gj)\stackrel{\bC^q}{\longrightarrow} \bH^{q+1}(\Gjp)\longrightarrow \cdots$$
is an exact sequence of $\Z[U]$--modules.}

\vspace{2mm}

The first 3 terms of the exact sequence (i.e. the $\bH^0$--part) were already used in \cite{OSzP}
(see also \cite{NOSZ}),
and the existence of the long exact sequence was already proved over $\Z_2$--coefficients in \cite{JG}.
Here we
establish its validity over $\Z$. In the proof we not only find the correct sign--modifications, but we also
replace some key arguments. (Nevertheless, the proof follows the main steps of \cite{JG}.)

For negative definite graphs (i.e. when $\Gjp$, hence $\Gamma$ and $\Gj$ too are negative definite), the above
exact sequence has some important additional properties. By general theory (cf. \cite{NLC}), if $\Gamma$ is negative definite then $\bH^0(\Gamma)$ contains a canonical submodule $\bt$ and one has a direct sum decomposition
 $\bH^0=\bt\oplus \bH^0_{red}$. $\bt$ is the analogue of the image of $HF^\infty$ in $HF^+$ in the Heegaard--Floer theory. On the other hand, in Heegaard--Floer theory
 by a result of Ozsv\'ath and Szab\'o, the operator $\bC^0$ restricted on $\bt(\Gj)$
 is zero (this follows from the fact that the corresponding cobordism connecting $\Gj$ and $\Gjp$ is coming from a  non--negative definite surgery). The analogue of this result is Theorem~\ref{ES2}:

\vspace{2mm}

\noindent {\bf Theorem B.} {\it
Assume that $\Gjp$ is negative definite. Consider the exact sequence
$$0\longrightarrow \bH^0(\Gjp)\stackrel{\bA^0}{\longrightarrow} \bH^0(\Gamma)\stackrel{\overline{\bB}^0}{\longrightarrow} \bH^0(\Gj)\stackrel{\bC^0}{\longrightarrow} \bH^{1}(\Gjp)\longrightarrow \cdots$$
and the canonical submodule  $\bt(\Gj)$ of \, $\bH^0(\Gj)$. Then
 the restriction $\bC^0|\bt(\Gj)$ is zero.}

\vspace{2mm}

Using the  above results one proves for a graph $\Gamma$
 {\it with at most  $n$ bad vertices} (for the definition, see
\ref{badver}) the vanishing $\bH^q_{red}(\Gamma)=0$ for any $q\geq n$
(where $\bH^q_{red}=\bH^q$ for $q>0$).

$\bH^*(\Gamma)$ has a natural direct sum decomposition indexed by the set of
the characteristic element classes
(in the case of the Heegaard--Floer, or Seiberg--Witten theory, they
correspond to the $spin^c$--structures of $M(\Gamma)$). Namely,
$\bH^*(\Gamma)=\oplus_{[k]} \bH^*(\Gamma,[k])$.
 In the exact sequence of Theorem A the operators mix these classes. Theorem \ref{ESk}
 provides an exact sequence which
connects the lattice cohomologies of $\Gamma $ and $\Gj$ with fixed (un--mixed) characteristic element classes.
 More precisely, for any characteristic element $k$ of $\Gamma$, we define a
 $\Z[U]$--module $\{\bH_{rel}^q(k)\}_{q\geq 0}$, the  {\it relative lattice cohomology} associated with
 $(\Gamma,j_0,k)$. It has  finite $\Z$--rank and it
 fits in the following exact sequence:

\vspace{2mm}

\noindent {\bf Theorem C.} {\it
Assume that $\Gamma$ and $\Gamma\setminus j_0$ are non--degenerate.
One has   a long exact sequence of \, $\Z[U]$--modules:
$$\cdots \longrightarrow \bH^q_{rel}(k)\stackrel{\bA^q_{rel}}{\longrightarrow} \bH^q(\Gamma,[k])\stackrel{\bB^q_{rel}}{\longrightarrow} \bH^q(\Gj,[R(k)])
\stackrel{\bC^q_{rel}}{\longrightarrow} \bH^{q+1}_{rel}(k)\longrightarrow \cdots$$
where $R(k)$ is the restriction of $k$ (and the operators also depend on the choice of the representative $k$).}

\vspace{2mm}

The existence of such a long exact sequence is predicted and motivated by the surgery formula for the
Seiberg--Witten invariant established in \cite{BN}: the results of section \ref{s:REL} resonate perfectly 
with the corresponding statements of Seiberg--Witten theory. This allows us to compute the
Euler characteristic of the relative lattice cohomology in terms of the Seiberg--Witten invariants associated with $M(\Gamma)$ and
$M(\Gamma\setminus j_0)$.

\section{Notations and preliminaries}
\subsection{Notations} First we will introduce the needed notations regarding plumbing graphs and
we recall the definition of the {\it lattice cohomology} from \cite{NLC}.
Since in [loc.cit.] the graphs were {\it connected} and {\it negative definite} (as the
`normal' plumbing representations of isolated complex  surface singularity links), at the beginning
{\it we will start with these assumptions}.

Let $\Gamma$ be such a  plumbing graph
with vertices $\calj$ and edges $\cale$; we set $|\calj|=s$.
 and we fix an order on $\calj$.
$\Gamma$ can also be codified in the lattice $L$, the free $\Z$--module generated by $\{E_j\}_{j\in \calj}$ and the
`intersection form' $\{(E_i,E_j)\}_{i,j}$, where $(E_i,E_j)$ for $i\not=j$ is 1 or 0 corresponding to the fact that
$(i,j)$ is an edge or not; and $(E_i,E_i)$ is the decoration of the vertex $i$, usually denoted by
$e_i$. (The graph is negative definite if this form is so.)
 The graph $\Gamma$ may have  cycles, but we will assume that  all the genus decorations are zero (i.e.
 we plumb $S^1$--bundles over $S^2$). The associated plumbed 3--manifold
$M(\Gamma)$ is not necessarily a  rational homology sphere, this happens exactly when
 the graph  is a tree.  Let $L'$ be the dual lattice
$\{l'\in L\otimes \Q\, :\, (l',L)\subset \Z\}$; it is generated by $\{E_j^*\}_j$, where $(E^*_j,E_i)=-\delta_{ij}$ (the {\it negative} of the Kronecker--delta). Moreover, $$Char:=\{k\in L'\,:\,
\chi_k(l):=-\frac{1}{2}(k+l,l)\in\Z \ \ \mbox{for all $l\in L$}\}$$ denotes the set of characteristic elements of $L$ (or $\Gamma$).

As usual (following Ozsv\'ath and Szab\'o), $\et_0^+$ denotes the $\Z[U]$--module
$\Z[U,U^{-1}]/U\Z[U]$ with grading $\deg(U^{-d})=2d$ ($d\geq 0$).
More generally, for any  $r\in\Q$ one defines $\et^+_r$, the same module as $\et_0^+$, but graded
(by $\Q$) in such a way that the $d+r$--homogeneous elements of $\et^+_r$ are isomorphic with the
$d$--homogeneous elements of $\et_0^+$.
(E.g., for $m\in \Z$,
$\et_{2m}^+=\Z[U,U^{-1}]/U^{-m+1}\Z[U]$.)

\subsection{The lattice cohomology associated with $k\in Char$ \cite{NLC}.}\label{s:LC}
$L\otimes \R=\Z^s\otimes_\Z \R$ has a natural cellular decomposition into cubes. The
set of zero--dimensional cubes is provided  by the lattice points
$L$. Any $l\in L$ and subset $I\subset {\mathcal J}$ of
cardinality $q$  defines a $q$--dimensional cube, which has its
vertices in the lattice points $(l+\sum_{j\in I'}E_j)_{I'}$, where
$I'$ runs over all subsets of $I$. On each such cube we fix an
orientation. This can be determined, e.g.,  by the order
$(E_{j_1},\ldots, E_{j_q})$, where $j_1<\cdots < j_q$, of the
involved base elements $\{E_j\}_{j\in I}$. The set of oriented
$q$--dimensional cubes defined in this way is denoted by $\calQ_q$
($0\leq q\leq s$).

Let $\calC_q$ be the free $\Z$--module generated by oriented cubes
$\square_q\in\calQ_q$. Clearly, for each $\square_q\in \calQ_q$,
the oriented boundary $\partial \square_q$ has the form
$\sum_k\varepsilon_k \, \square_{q-1}^k$ for some
$\varepsilon_k\in \{-1,+1\}$. Here, in this sum, we write only
those $(q-1)$--cubes which appear with non--zero coefficient.
One sees that  $\partial\circ\partial=0$,  but, obviously, the
homology of the chain complex $(\calC_*,\partial)$ is trivial:
it is just the homology of $\R^s$.  In order to get
a more interesting (co)homology, one needs to consider a {\em weight
functions} $w:\calQ_q\to \Z$ ($0\leq q\leq s$).
In the present case this will be defined, for each $k\in Char$ fixed, by
$$w(\square_q):=\max\{\chi_k(l)\, :\, \ \mbox{$l$ \ is a vertex of $\square_q$}\}.$$

Once the weight function is defined, one considers
 $\calF^q$, the set of morphisms  $\Hom_{\Z}(\calC_q,\et^+_0)$
  with finite support on $\calQ_q$.
Notice that $\calF^q$ is, in fact, a $\Z[U]$--module by
$(p*\phi)(\square_q):=p(\phi(\square_q))$ ($\phi\in \calF^q$, $p\in \Z[U]$).
Moreover, $\calF^q$ has a $2\Z$--grading: $\phi\in \calF^q$ is
homogeneous of degree $2d\in\Z$ if for each $\square_q\in\calQ_q$
with $\phi(\square_q)\not=0$, $\phi(\square_q)$ is a homogeneous
element of $\et^+_0$ of degree $2d-2\cdot w(\square_q)$.

Next, one defines $\delta:\calF^q\to \calF^{q+1}$. For
this, fix $\phi\in \calF^q$ and we show how $\delta(\phi)$ acts on
a cube $\square_{q+1}\in \calQ_{q+1}$. First write
$\partial\square_{q+1}=\sum_k\varepsilon_k \square ^k_q$, then
set
\begin{equation}\label{eq:delta}
(\delta(\phi))(\square_{q+1}):=\sum_k\,\varepsilon_k\,
U^{w(\square_{q+1})-w(\square^k_q)}\, \phi(\square^k_q).
\end{equation}
One verifies that  $\delta\circ\delta=0$, i.e.
$(\calF^*,\delta)$ is a cochain complex (with $\delta$
homogeneous of degree zero);
its homology is denoted by $\{\bH^q(\Gamma,k)\}_{q\geq 0}$.

 $(\calF^*,\delta_w)$ has a natural
augmentation too. Indeed, set $m_k:=\min_{l\in L}\{\chi_k(l)\}$.
Then one defines the $\Z[U]$--linear map
$\epsilon:\et^+_{2m_k}\longrightarrow \calF^0$ such that
$\epsilon(U^{-m_k-s})(l)$ is the class of $U^{-m_k+\chi_k(l)-s}$
in $\et^+_0$ for any integer $s\geq 0$. Then
 $\epsilon$ is injective and homogeneous of degree
zero, $\delta\circ\epsilon=0$.
The homology of  the augmented cochain complex
$$0\longrightarrow\et^+_{2m_k}\stackrel{\epsilon}{\longrightarrow}
\calF^0\stackrel{\delta}{\longrightarrow}\calF^1
\stackrel{\delta}{\longrightarrow}\ldots$$ is called the {\em
reduced lattice cohomology} $\bH_{red}^*(\Gamma,k)$. For any
$q\geq 0$, both $\bH^q$ and $\bH_{red}^q$ admit an induced graded
$\Z[U]$--module structure and $\bH^q=\bH^q_{red}$ for $q>0$.

Note that $\epsilon$ provides a canonical embedding of $\et^+_{2m_k}$ into
$\bH^0$.  Moreover,  one has a graded
$\Z[U]$--module isomorphism
$\bH^0=\et^+_{2m_k}\oplus\bH^0_{red}$, and
$\bH^*_{red}$ has finite $\Z$--rank.
\begin{remark}
For the definition of the lattice cohomology for more general
weight functions and graphs with non--zero genera, see \cite{NLC}.
\end{remark}

\subsection{Reinterpretation of the lattice cohomology}
If $k'=k+2l$ for some $l\in L$ then $\bH^*(\Gamma,k)$ and
$\bH^*(\Gamma,k')$ are isomorphic up to a degree--shift, cf. \cite[(3.3)]{NLC}.
In fact, all the modules $\{\bH^*(\Gamma,k)\}_{k}$ can be packed into only one object,
more in the spirit of \cite{OSzP}. This  was used in \cite{JG} too.

In this way, for any fixed $k\in Char$, $L$ is identified with the sublattice $k+2L\in Char$,
and with the notation $l':=k+2l$ one has $\chi_k(l)=-\frac{1}{8}(l',l')+\frac{1}{8}(k,k)$. In particular,
up to a shift in degree, for each fixed $k$, $l\mapsto \chi_k(l)$ and $l'\mapsto -\frac{1}{8}(l',l')$
define the same weight--function (on the  cubes, see below), hence the same cohomology.
In fact, we will modify even this  weight function by $s/8$
(in this way the blowing up will induce a degree preserving isomorphism, cf. \S\ref{s:3}),  and set:
\begin{equation}\label{eq:q}
\q: Char\to \Q, \ \ \ \ \ \ \q(k):=-\frac{k^2+s}{8} \ \ \ \ (s=|\calj|).
\end{equation}
The $q$--cubes  $\square_{q}\in \calQ_{q}$ are associated with
pairs $(k,I)\in Char\times \calP(\calj)$, $|I|=q$, (here $\calP(\calj)$ denotes the power set of $\calj$),
and have the form $\{k+2\sum_{j\in I'}E_j)_{I'}$, where
$I'$ runs over all subsets of $I$.
The weights are defined by
\begin{equation}
w(\square_q)=w((k,I))=\max_{I'\subset I}\big\{\, \q(k+2\sum_{j\in I'}E_j)\,\big\}.
\end{equation}
Moreover,  $\calF^q $ are elements of $\Hom_{\Z}(\calQ_q,\et^+_0)$
with finite support.
Similarly as above,  $\calF^q$ is  a $\Z[U]$--module with a $\Q$--grading: $\phi\in \calF^q$ is
homogeneous of degree $r$ if for each $\square_q\in\calQ_q$
with $\phi(\square_q)\not=0$, $\phi(\square_q)$ is a homogeneous
element of $\et^+_0$ of degree $r-2\cdot w(\square_q)$.

It is convenient to consider the module of (infinitely supported)
homological cycles too: let $\calF_q$ be the direct product
 of  $\Z_{\geq 0}\times \calQ_q$ copies of $\Z$ (considered already in \cite{OSzP} for $q=0$). We write the pair $(m,\square)$ as $U^m \square$.
 $\calF_q$ becomes a $\Z[U]$--module by $U(U^m \square)=U^{m+1} \square$.
 Clearly $\calF^q=\Hom_{\Z[U]}(\calF_q,\et^+_0)$, i.e. $\phi(U\square)=U\phi(\square)$ for any $\phi$.

$\delta:\calF^q\to \calF^{q+1}$ is defined as in (\ref{eq:delta}) using the new weight--function,
 or by $\delta(\phi)(\square)=\phi(\partial (\square))$, where for $\square=(k,I)=(k,\{j_1,\ldots,j_q\})$ one has:
\begin{equation}\label{eq:partial}
\partial (k,I)=\sum_{l=1}^q (-1)^l\big( \,
U^{w(k,I)-w(k,I\setminus j_l)} (k,I\setminus j_l)-U^{w(k,I)-w(k+2E_{j_l},I\setminus j_l)}
(k+2E_{j_l},I\setminus j_l)\, \big).
\end{equation}
The cohomology of $(\calF^*,\delta)$ is denoted by $\bH^*(\Gamma)$.
Since the vertices of a cube belong to the same class $Char/2L$ (where a class has the form
$[k]=\{k+2l\}_{l\in L}\subset Char$), $\bH^*(\Gamma)$ has a natural direct sum decomposition
$$\bH^*(\Gamma)=\oplus_{[k]\in Char/2L}\ \bH^*(\Gamma,[k]).$$
In fact, if $[k_1]=[k_2]$ then $w(k_1)-w(k_2)\in \Z$.

Since $\Gamma $ is negative definite, for each class $[k]\in Char/2L$ one has a well--defined
rational number
$$d[k]:=-\max _{k\in[k]} \frac{k^2+s}{4}=2\cdot \min_{k\in[k]} \q(k).$$
Then (cf. \ref{s:LC}) one has a direct sum decomposition:
\begin{equation}\label{eq:et}
\bH^0(\Gamma,[k])=\et^+_{d[k]}\oplus \bH^0_{red}(\Gamma,[k]).
\end{equation}
Sometimes we write $\bt:=\oplus_{[k]}\et^+_{d[k]}$ for the canonical submodule
of $\bH^0$, which satisfies  $\bH^0=\bt\oplus \bH^0_{red}$.

Following  \cite{NOSZ,NLC,NSW} we define the Euler characteristic of $\bH^*(\Gamma,[k])$ by
\begin{equation}\label{eq:EU}
eu(\bH^*(\Gamma,[k])):=-d[k]/2+\sum_{q\geq 0}(-1)^q\rank_\Z\, \bH^q_{red}(\Gamma,[k]).
\end{equation}

\begin{remark}\label{re:1}
 For any $\phi \in\calF^q$ and $\ell\geq 0$ set $U^{-\ell}*\phi\in \calF^q$ defined as follows: if $\phi(\square)=\sum_{m\geq 0}
 a_{m,\square}U^{-m}$, then $(U^{-\ell}*\phi)(\square)=\sum_{m\geq 0}
 a_{m,\square}U^{-m-\ell}$. Notice that $U(U^{-1}*\phi)=\phi$ (but, in general,
  $U^{-1}*(U\phi)\not=\phi$).

 For any $\square\in\calQ_q$ let  $\square^\veeK$ be that element of $\calF^q$ which sends $\square$ in
 $1\in\et^+_0$ and any other element into zero.   Then any element of $\calF^q$ is a finite $\Z$--linear combination of elements of type $U^{-\ell}*\square^\veeK$.
\end{remark}

\subsection{Generalizations.  Graphs which are not negative definite.}\label{2.4}
 Since the definition of the lattice cohomology
is purely algebraic/combinatorial, its definition can be considered for any graph, or even for any lattice $L$  with fixed base elements $\{E_j\}_j$.
Nevertheless, in this note we assume that the graphs are non--degenerate.
Of course, doing this generalization, some of the properties of the lattice cohomology associated with
connected negative definite graphs will not survive. Here we wish to point out some of the differences.

First, we notice that dropping the connectedness of the graph basically has no effect (this fact already was used in the main inductive proofs of \cite{OSzP} and \cite[(8.3)]{NOSZ}).  Dropping the negative definiteness is more serious. In order to explain this, we will introduce the following spaces: for any fixed class $[k]$
and any real number $r$ we define $S_r$ as the union of all cubes $\square\in \calQ_q $ with
all vertices in $[k]$ and  $w(\square)\leq r$.

If $\Gamma$ is negative definite then $S_r$ is compact for any $r$. In particular, the weight function takes its minimum on it, and for each class $[k]=Char/2L$  a decomposition $\bH^0(\Gamma,[k])=\et^+_{d([k])}\oplus \bH^0_{red} (\Gamma,[k])$ can be defined. This property will not survive for general graphs. In fact, if the components of  $S_r$ are not compact, then  we even might have the vanishing $\bH^0(\Gamma)=0$ (and this can happen simultaneously with the non--vanishing of
$\bH^1$), see e.g. (\ref{+1}). Also, in \cite[Theorem 3.1.12]{NLC} the lattice cohomology is recovered from
the simplicial cohomology of the spaces $S_r$. In the general case,
the adapted  version of that theorem (with similar proof), valid for any lattice,
states that the same formula is valid once if we replace the cohomology groups $H^q(S_r,\Z)$ by the
cohomology groups with compact support $H^q_c(S_r,\Z)$.
Namely:
$$\bH^*(\Gamma,[k])\simeq\bigoplus_{r\in w(k)+\Z}\, H^*_c(S_r,\Z).$$
We wish to emphasize, that replacing the cohomology with cohomology with compact support
is a conceptual modification: the two groups have different functorial properties.

In fact,  exactly this second point of view is determinative  in the
definition of the lattice cohomology (and in the definition of the maps in the surgery formula treated next):
basically we mimic the infinitely supported homology and the (dual) cohomology with compact support,
and the corresponding functors associated with them.

Since the surgery exact sequence considered in the next section is valid for any non--degenerate
lattice, it is very convenient to extend the theory to arbitrary graphs (or, at least for non--degenerate ones)
in order to have a larger flexibility for computations. Nevertheless, we have to face the following crucial problem. An oriented  plumbed 3--manifold $M(\Gamma)$ has many different plumbing representations $\Gamma$.
They are connected by the moves of the (oriented) plumbing calculus. For negative definite graphs the only moves are the blowups (and their inverses) by (rational) $(-1)$-vertices. In \cite[(3.4)]{NLC} it is proved that the lattice cohomology is stable with respect to these moves, see also  \S\ref{s:3} here. On the other hand, if we
enlarge our plumbing graphs, this stability condition will not survive: the same 3--manifold can be represented
by many different lattices with rather different lattice cohomologies.

More precisely, for negative definite graphs one conjectures (cf. \cite{NLC}) that from the lattice cohomology one can recover (in a combinatorial way) the Heegaard--Floer homology of Ozsv\'ath--Szab\'o. In particular, the lattice cohomology carries a geometric meaning depending only on $M(\Gamma)$. This geometric meaning is lost in the context of general graphs (or, at least, it is not so direct).

\begin{example} \label{+1}
$S^3$ can be represented by a graph with one vertex, which has decoration $-1$. Computing the lattice cohomology of this graph we get $\bH^q=0$ for $q\not=0$ and $\bH^0=\bH^0_{red}=\et^+_{0}$.
(This is the Heegaard--Floer homology $HF^+(S^3)$.) On the other hand, it is instructive to compute the lattice cohomology of another graph, which has one vertex, but now with decoration $+1$. Its lattice cohomology is
$\bH^q=0$ for $q\not=1$ and $\bH^1= \et^+_{-1/2}$. This graph also represent $S^3$, and the two graphs can be connected by a sequence of non--empty graphs and $\pm 1$--blow ups and downs. In particular, we conclude that
the lattice cohomology is not stable with respect to blowing up/down $(+1)$--vertices.
\end{example}

\begin{remark}\label{re:weightS}
Replacing negative definite graphs with arbitrary non--degenerate ones we can also adopt the
definition of the weight function (\ref{eq:q}) by modifying into
\begin{equation}\label{eq:q2}
\q_\dagger: Char\to \Q, \ \ \ \ \ \ \q_\dagger(k):=\frac{-k^2+3\sigma+2s}{8} \ \ \ \ (\sigma=\mbox{signature of }(\,,\,),\ \ s=|\calj|),
\end{equation}
a more common expression associated with (4--manifold intersection) forms which are not negative definite.
In this note we will not do this, but the interested reader preferring this expression might shift
all the weights below by $3(\sigma+s)/8$.
\end{remark}

\section{Blowing up $\Gamma$}\label{s:3}

\subsection{}\label{3.1}
Since in the main construction we will need
some of the operators induced by blowing down, we will make explicit the involved morphisms.

The next  discussion provides a new proof of the stability of the lattice cohomology
in the case when we blow up a vertex (for the old proof see \cite[(3.4)]{NLC}).

{\it Starting from now, the graph is neither necessarily connected, nor necessarily negative definite.
Nevertheless, we will assume that the intersection form  is non--degenerate. }

We assume that $\Gamma'$ is obtained from
$\Gamma$ by `blowing up the vertex $j_0$'.  More
precisely, $\Gamma'$ denotes a graph with one more vertex and one
more edge: we glue to the vertex $j_0$ by the new
edge the new vertex $E_{new}$ with decoration $-1$ (and genus 0), while the
decoration of $E_{j_0}$ is modified from $e_{j_0}$ into
$e_{j_0}-1$,  and we keep all the other decorations.
We will use the notations $L(\Gamma),\ L(\Gamma'),\ L'(\Gamma), \ L'(\Gamma')$.
Let $\q':Char(\Gamma')\to\Q$ be the weight function of
$\Gamma'$ defined similarly as in (\ref{eq:q}).
(We may use the following convention for the ordering of the indices: if $j\not=j_0$, then
$j<j_0<j_{new}$.)
The following facts can be verified:

\begin{bekezdes}\label{bullet:1}
 Consider the maps $\pi_*:L(\Gamma')\to
L(\Gamma)$ defined by $\pi_*(\sum x_jE_j+x_{new}E_{new})=\sum
x_jE_j$, and $\pi^*:L(\Gamma)\to L(\Gamma') $  defined by
$\pi^*(\sum x_jE_j)=\sum x_jE_j +x_{j_0}E_{new}$. Then
$(\pi^*x,x')_{\Gamma'}=(x,\pi_*x')_\Gamma$.  This shows that
$(\pi^*x,\pi^*y)_{\Gamma'}=(x,y)_\Gamma$ and $(\pi^*x,E_{new})_{\Gamma'}=0$ \ for any
$x,y\in L(\Gamma)$. Both $\pi_*$ and $\pi^*$ extend over $L\otimes \Q$ and $L'$.
\end{bekezdes}

\begin{bekezdes}\label{bullet:2} Set the  (nonlinear)
map: $c:L'(\Gamma)\to L'(\Gamma')$,
$c(l'):=\pi^*(l')+E_{new}$. Then $c(Char(\Gamma))\subset
Char(\Gamma')$ and $c$  induces an isomorphism between the orbit
spaces $Char(\Gamma)/2L(\Gamma)\simeq Char(\Gamma')/2L(\Gamma')$.
Moreover, for any $k\in Char(\Gamma)$ and
$k':=c(k)$ one has $\q'(k')=\q(k)$.
\end{bekezdes}

\begin{bekezdes}\label{bullet:3}
$\pi_*(Char(\Gamma'))\subset Char(\Gamma)$. For any $k'\in Char(\Gamma')$ write
$k:=\pi_*(k')$. Then $k'=\pi^*(k)+aE_{new}$ for some odd integer $a$, and
$\q'(k')-\q(k)=\frac{a^2-1}{8}\in \Z_{\geq 0}$.
\end{bekezdes}

The maps $c$ and $\pi_*$ can be extended to the level of cubes and complexes
as follows.

\begin{bekezdes}
For any $q\geq 0$ and $\square=(k',I)\in \calQ_q(\Gamma')$ one defines
$\pi_*((k',I)):= (\pi_*(k'),I)\in \calQ_q(\Gamma)$, provided that $j_{new}\not\in I$.
By (\ref{bullet:3}) one has $w_{\Gamma'}(\square)-w_\Gamma(\pi_*(\square))\geq 0$.
This defines a homological morphism $\pi^h_*:\calF_q(\Gamma')\to \calF_q(\Gamma)$ by
\begin{equation}
\pi^h_*(\square)=\left\{
\begin{array}{ll}
U^{w_{\Gamma'}(\square)-w_\Gamma(\pi_*(\square))} \, \pi_*(\square) & \mbox{if $j_{new}\not\in I$},\\
0 & \mbox{else}.\end{array}
\right.
\end{equation}
Using (\ref{eq:partial}) one verifies that $\pi_*^h\circ \partial =\partial \circ \pi_*^h$, hence $\pi_*^h$ is
morphism of homological complexes. In particular, its dual $\pi_*^c:\calF^q(\Gamma)\to \calF^q(\Gamma')$,
defined by $\pi_*^c(\phi)=\phi\circ \pi_*^h$, satisfies $\pi_*^c\circ \delta =\delta \circ \pi_*^c$, hence it is a
morphism of complexes as well.
\end{bekezdes}

\begin{bekezdes}
Before we extend $c$ to the level of complexes, notice that for any $k\in Char(\Gamma)$
\begin{equation*}
c(k+2E_j)=\left\{\begin{array}{ll}
c(k)+2E_j & \mbox{if $j\not=j_0$},\\
c(k)+2E_j+2E_{new} & \mbox{if $j=j_0$}.\end{array}\right.\end{equation*}
In particular, the pair $c(k)$ and $c(k+2E_{j_0})$ does not form a 1--cube in $\Gamma'$. Hence,
the application $(k,I)\mapsto (c(k),I)$, sending the vertex $k+2\sum_{j\in I'}E_j$ into
$c(k)+2\sum_{j\in I'}E_j$,
{\it does not} commute with the boundary operator.
The `right' operator  $c^h:\calF_q(\Gamma)\to \calF_q(\Gamma')$ is defined by
\begin{equation}
c^h((k,I))=\left\{\begin{array}{ll}
(c(k),I) & \mbox {if $j_0\not\in I$},\\
(c(k),I)+U^{w(k,I)-w(k+2E_{j_0},I_0)}(c(k)+2E_{j_0},I_0\cup j_{new}) & \mbox{if $I=I_0\cup j_{0}$, \
 $j_0\not\in I_0$}.\end{array}\right.
\end{equation}
In fact, by (\ref{eq:for1}) below,  $w(k+2E_{j_0},I_0)$ above can be replaced by
$w' (c(k)+2E_{j_0},I_0)$, or even by  $w' (c(k)+2E_{j_0},I_0\cup j_{new})$.
By (\ref{bullet:2}) and a computation one gets (where in the third line
$I_0'$ is any subset of $\calj$ with $j_0\not\in I_0'$):
\begin{equation}
\begin{array}{rl}
 w((k,I))= & w'((c(k),I)),\label{eq:for1}\\
w((k+2E_{j_0},I_0))=&w'((c(k)+2E_{j_0},I_0)),\\
\q'(c(k)+2\sum_{j\in I'_0}E_j+2E_{j_0})=&
\q'(c(k)+2\sum_{j\in I'_0}E_j+2E_{j_0}+2E_{new}).\end{array}\end{equation}
 These and a (longer) computations shows that
$c^h$ commutes with the boundary operator $\partial$.
\end{bekezdes}

\begin{bekezdes}
Using $\pi_*\circ c$, the definitions and the first equation of (\ref{eq:for1}) one gets
$\pi_*^h\circ c^h=id_{\calF_*(\Gamma)}$. On the other hand,
$c^h\circ \pi_*^h$ is not the identity, but it is homotopic
to $id_{\calF_*(\Gamma')}$. Indeed, we define the homotopy operator $K:\calF_*(\Gamma')\to
\calF_{*+1}(\Gamma')$ as follows. Write any $k$ as $c\pi_*(k)+2aE_{new}$ for some $a\in \Z$.
Then define $K((k,I))$ as 0 if either  $j_{new}\in I$ or $a=0$. Otherwise take for $K((k,I))$:
\begin{equation*}
sign(a)\cdot \sum \, U^{w(k,I)-w(c\pi_*(k)+2lE_{new},I\cup j_{new})}\, (c\pi_*(k)+2lE_{new},I\cup j_{new}),
\end{equation*}
where the summation is over $l\in \{0,1,\ldots,a-1\}$ if $a>0$ and
$l\in\{a,\ldots, -1\}$ if $a<0$. (The exponents are non--negative because of (\ref{bullet:3}).)
Then, again by a computation,  $\partial \circ K-K\circ \partial=id-c^h\circ \pi_*^h$.

In particular, $\pi_*^h$ and $c^h$ induce (degree preserving)
isomorphisms of the corresponding lattice cohomologies. In the sequel the operator $\pi_*^h$
will be crucial.
\end{bekezdes}

\section{Comparing $\Gamma$ and $\Gj$}\label{s:4}

\subsection{Notations, remarks}\label{ss:notrem} We consider a non--degenerate graph $\Gamma$ as
in \ref{3.1}, and we fix one of its vertices
$j_0\in \calj$. The new graph
$\Gj$ is  obtained from $\Gamma$ by deleting the vertex $j_0$ and all its adgacent
edges. We will denote by $L(\Gamma), L'(\Gamma), L(\Gj), L'(\Gj)$ the corresponding lattices.

The operator $i:L(\Gj)\to L(\Gamma)$, $i(E_{j,\Gj})=E_{j,\Gamma}$ identifies $L(\Gj)$ with a sublattice of
$L(\Gamma)$. The dual operator (restriction) is $R:L'(\Gamma)\to L'(\Gj)$, $R(E^*_{j,\Gamma})=E^*_{j,\Gj}$ for
$j\not=j_0$ and  $R(E^*_{j_0,\Gamma})=0$. It satisfies $(i_\Q(x),y)_\Gamma=(x,R(y))_{\Gj}$ for any
$x\in L'(\Gj)$ and $y\in L'(\Gamma)$. (Here, $E^*_{j,\Gamma}$ respectively $E^*_{j,\Gj}$ are the usual dual generators of $L'$
considered in $\Gamma$, respectively in $\Gj$.)

Recall  that $l'=\sum_ja_jE^*_{j,\Gamma}$ is characteristic if and only if
$a_j\equiv e_j\ (mod\ 2)$. In particular,
$R$ sends characteristic elements into characteristic elements. On the other hand, $i_\Q$ not necessarily preserves
characteristic elements.

Although $R(Char(\Gamma))\subset Char(\Gj)$, this operator cannot be extended 
to the level of cubes. Indeed, notice that
\begin{equation}\label{eq:R}
R(E_{j_0})=-\ \sum _{(j,j_0)\in \cale_\Gamma} \ E^*_{j,\Gj},
\end{equation}
where the sum runs over the adjacent vertices of $j_0$ in $\Gamma$. In particular, the endpoints of the 1--cube
$(k,j_0)$ are sent into $R(k)$ and $R(k)+2R(E_{j_0})$,  which cannot be expressed as combination of
1--cubes in $\Gj$.
In the next subsection we will consider another operator, which can be extended to the level of
cubes (and which, in fact, operates in a different direction than $R$, cf. \ref{2.4}).

\subsection{The $B$--operator.} Consider $b:L'(\Gj)\to L'(\Gamma)$ defined by
$$\sum_ja_jE^*_{j,\Gj}\mapsto \sum_ja_jE^*_{j,\Gamma}.$$
Clearly, if $k\in Char(\Gj)$ then $b(k)+a_{j_0}E^*_{j_0,\Gamma}\in Char(\Gamma)$ for any
$a_{j_0}$ with $a_{j_0}\equiv e_{j_0} \ (mod\ 2)$. In order to see how it operates on cubes, notice that
in $\Gj$ for any $j\in\calj(\Gj)$ one has
$$E_{j,\Gj}=-e_j\, E^*_{j,\Gj}-\sum_{(i,j)\in\cale_{\Gj}}\, E^*_{i,\Gj},$$
and a similar relation holds in $\Gamma$ too. Therefore, one gets that
\begin{equation}\label{eq:shift}
b(E_j)=E_j+ (E_{j_0},E_j)_\Gamma \cdot E^*_{j_0,\Gamma}.
\end{equation}
Therefore,
\begin{equation}\label{biq}
b(l')=i_\Q(l')+(E_{j_0},i_\Q(l'))_\Gamma\cdot E^*_{j_0,\Gamma}.\end{equation}
In particular, $b$ is a `small' modification of $i_\Q$, but this modification is enough to
extend it to the level of cubes. Although the vertices of $(k,I)$ are not sent to the vertices of a cube
(provided that $I$ contains elements adjacent to $j_0$ in $\Gamma$), cf. (\ref{eq:shift}),
nevertheless if we consider all the shifts in the direction of the `error' of (\ref{eq:shift}) we get
 a well--defined  operator
$$B_q:\calF_q(\Gj)\to \calF_q(\Gamma)$$
given by
\begin{equation}\label{biq2} B_q((k,I)):=\sum_{a_{j_0}\equiv e_{j_0}\, (mod\, 2)}\
(b(k)+a_{j_0}E^*_{j_0,\Gamma}, I).\end{equation}
(Here we keep the notation `$B$', since this notation was used in similar contexts, as in
\cite{OSzP,NOSZ,JG}.)

$B_q$ induces a morphism $B^q:\calF^q(\Gamma)\to \calF^q(\Gj)$ via $(B^q\phi)(\square)=\phi(B_q\square)$. A straightforward (slightly long) computation shows that $B_q$ commutes with the boundary operator $\partial$, hence
$B^*\circ \delta=\delta\circ B^*$ too. In particular, one gets a well--defined morphism of $\Z[U]$--modules
$\bB^*:\bH^*(\Gamma)\to \bH^*(\Gj)$.

\subsection{The sign--modified $B$--operator.}\label{4.3}
 In the exact sequences considered in the next section we will
need to modify the $B$--operator by a sign (compare with the end of \S 2 of \cite{OSzP}, where
the case $q=0$ is discussed). The definition depends on some choices.

For each $L$--orbit $[k]:=k+2L(\Gj)\subset Char(\Gj)$ we will fix a representative $r_{[k]}\in [k]$.
Hence, for any characteristic element $k$, with orbit $[k]$, $k-r_{[k]}\in 2L(\Gj)$, hence
$(k-r_{[k]},E^*_{j,\Gj})\in 2\Z$ for any $j\not=j_0$. In particular,
$(k-r_{[k]},R(E_{j_0}))\in 2\Z$ for $R(E_{j_0})$ defined in (\ref{eq:R}).
 Then  the modified operator
$$\overline{B}_q:\calF_q(\Gj)\to \calF_q(\Gamma)$$
is given by
\begin{equation}\overline{B}_q((k,I)):=\sum_{a_{j_0}\equiv e_{j_0}\, (mod\, 2)}\
(-1)^{n(k,a_{j_0})/2} \cdot
(b(k)+a_{j_0}E^*_{j_0,\Gamma}, I),\end{equation}
where $$n(k,a_{j_0}):= (k-r_{[k]},R(E_{j_0}))+a_{j_0}+e_{j_0}.$$
Then, again, $\overline{B}^*$ commutes with the boundary operator, hence
$\overline{B}^*:\calF^q(\Gamma)\to \calF^q(\Gj)$  defined by $(\overline{B}^q\phi)(\square)=
\phi(\overline{B}_q\square)$ commutes with $\delta$, hence it also defines a  $\Z[U]$--modules
morphism $\overline{\bB}^*:\bH^*(\Gamma)\to \bH^*(\Gj)$.

\begin{remark}
The morphisms $\bB^*$ and $\overline{\bB}^*$ do {\it not} preserve the gradings {\it neither} the
orbits $Char/2L$ (i.e. they do not split in direct sum with respect to these orbits).
\end{remark}

\section{The surgery exact sequence}

\subsection{Notations} Let $\Gamma$ be a non--degenerate graph  as above with $s=|\calj|\geq 2$.
We fix a vertex $j_0\in \calj$. Associated with $j_0$ we consider two other graphs, namely $\Gj$ and
$\Gjp$. The second one is obtained from $\Gamma$ by modifying the decoration $e_{j_0}$ of the vertex $j_0$
into $e_{j_0}+1$. In order to stay in our category of objects, we  will assume that both graphs
$\Gj$ and $\Gjp$ are non--degenerate.

Our goal is to establish  a long exact sequence  connecting the lattice cohomologies of
these three graphs, similar which is valid for the Heegaard--Floer cohomologies provided by
surgeries. As usual, in order to define a long exact sequence, we need first to determine a short exact
sequence of complexes.

The graphs $\Gamma$ and $\Gj$ will be connected by the $\overline{B}$--operator,
cf. \ref{4.3}. We define the `$A$'--operator connecting $\Gamma$
and $\Gjp$ as follows. Let $\Gamma^b$ be the graph obtained from $\Gamma$ by attaching a new vertex $j_{new}$ (via
a single new  edge) to $j_0$, such that the decoration of the new vertex is $-1$. Then, for the pair
$\Gamma$, $\Gamma^b$ (since $\Gamma=\Gamma^b\setminus j_{new}$), we can apply the construction and results
of  \S\ref{s:4}. In particular, we get  morphisms of complexes:
$B_*:\calF_*(\Gamma)\to \calF_*(\Gamma^b)$ and
$B^*:\calF^*(\Gamma^b)\to \calF^*(\Gamma)$.
Next, since $\Gjp$ obtained from $\Gamma^b$ by blowing down the new vertex $j_{new}$, by \S\,\ref{s:3}
we get morphisms of complexes:
$\pi_*^h:\calF_*(\Gamma^b)\to \calF_*(\Gjp)$ and
$\pi_*^c:\calF^*(\Gjp)\to \calF^*(\Gamma^b)$.
By composition we get  the $A$--operators:
$$A_*:=\pi_*^h\circ B_*: \calF_*(\Gamma)\to  \calF_*(\Gjp), \ \mbox{and } \
A^*:=B^*\circ \pi_*^c:\calF^*(\Gjp)\to \calF^*(\Gamma).$$
In particular, we can consider the short sequence of complexes
\begin{equation}\label{eq:SS}
0\to \calF^*(\Gjp)\stackrel{A^*}{\longrightarrow} \calF^*(\Gamma)
\stackrel{\overline{B}^*}{\longrightarrow}\calF^*(\Gj)\to 0.
\end{equation}
\begin{theorem}\label{Th1}
The short sequence of complexes (\ref{eq:SS}) is exact.
\end{theorem}
The proof is given in several steps.

\subsection{The injectivity of $\mathbf{A^q}$}\label{sINJ}
 Take an arbitrary  non--zero $\phi\in \calF^q(\Gjp)$. In order to prove
that $A^q\phi\not=0$, we need to find $f\in \calF_q(\Gamma)$ such that $\phi(A_q(f))\not=0$.
Let $N$ be the smallest non--negative integer such that $U^{N+1}\phi=0$. Then replacing $\phi$ by $U^N\phi$ we may assume that $U\phi=0$. In particular, in $f$ (or in $A_q(f)$) any term whose coefficient has the form
$U^n$ ($n>0$) is irrelevant.

Since $\phi\not=0$, there exists $(\bar{k},I)\in \calQ_q(\Gjp)$ with $\phi((\bar{k},I))\not=0$. Write
\begin{equation}\label{bark}
\bar{k}=\sum_{j\in\calj}a_jE^*_{j,\Gjp}+E^*_{j_0,\Gjp},
\end{equation}
with $a_j-e_j$ even for all $j\in\calj$.
Since $\phi$ is finitely supported, we may assume that
\begin{equation}\label{eq:min}
\mbox{$a_{j_0}$ is minimal with the property $\phi((\bar{k},I))\not=0$.}
\end{equation}
We wish to construct $f\in \calF_q(\Gamma)$ such that $\phi(A_q(f)-(\bar{k},I))=0$.
For the weight functions
in $\Gamma^b$ and $\Gjp$ we will use the notations $\q^b$ respectively $\q^+$.

First, set \begin{equation}\label{k}
k:=\sum_ja_jE^*_{j,\Gamma}\in Char(\Gamma).\end{equation}
 Then
$$B_0(k)=\sum_{a\equiv 1}k_a^b, \ \ \mbox{where} \ \
k_a^b=\sum_ja_jE^*_{j,\Gamma^b}+aE^*_{new}.$$
Set also
$$k_a:=\pi_*(k_a^b)=\sum_ja_jE^*_{j,\Gjp}+aE^*_{j_0,\Gjp}.$$
Notice that $k_1=\bar{k}$. Since $\pi^*k_a=
k^b_a-aE_{new}$, by
(\ref{bullet:3})
\begin{equation}\label{eq:qq}
\q^b(k^b_a)=\q^+(k_a)+(a^2-1)/8.\end{equation}
Moreover, $B((k,I))=\sum_{a\equiv 1}(k^b_a,I)$ and $\pi_*((k^b_a,I))=(k_a,I)$.

For any $I'\subset I$ write $E_{I'}=\sum_{j\in I'}E_j$. Let $\delta_{I'}$ be 1 if $j_0\in I'$, and zero
otherwise. Since
\begin{equation}\label{I'1}
\q^b(k^b_a+2E_{I'})-\q^b(k^b_a)=\frac{1}{2}\sum_{j\in I'}a_j-\frac{1}{2}(E_{I'},E_{I'})_{\Gamma^b},
\end{equation}
and
\begin{equation}\label{I'2}
\q^+(k_a+2E_{I'})-\q^+(k_a)=\frac{1}{2}\sum_{j\in I'}a_j +\frac{1}{2}\delta_{I'}a-\frac{1}{2}(E_{I'},E_{I'})_{\Gjp},
\end{equation}
we get
\begin{equation}\label{I'3}
\q^b(k^b_a+2E_{I'})-\q^+(k_a+2E_{I'})=\frac{a^2-1}{8}-\delta_{I'}\cdot \frac{a-1}{2}.
\end{equation}
Assume that $I\not\ni j_0$. Then, by (\ref{I'3}), we get that $w^b(k^b_a,I)-w^+(k_a,I)=(a^2-1)/8$, hence
\begin{equation}\label{A}
A_q((k,I))=\sum_{a\equiv 1, \, a<0} U^{\frac{a^2-1}{8}}(k_a,I)+(\bar{k},I)+
\sum_{a\equiv 1, \, a\geq 3} U^{\frac{a^2-1}{8}}(k_a,I).\end{equation}
Both  sums are killed by $\phi$, the first one  by (\ref{eq:min}), the second one by  $U\phi=0$, hence
$\phi(A_q((k,I))-(\bar{k},I))=0$.

Next, assume that $I\ni j_0$. In that case, again by
(\ref{I'3}), we get that
$w^b(k^b_a,I)-w^+(k_a,I)>0$  whenever $a\not\in\{-1,1,3\}$.
For the other three special values we have the following facts.
It is convenient to set  for each $k$ (represented as in (\ref{k})) the integer:
\begin{equation}
M(k):=\max_{I'}\, \{\,\sum_{j\in I'}a_j-(E_{I'},E_{I'})_{\Gamma^b}\,\}.
\end{equation}

$\bullet$ For $a=1$ one has $w^b(k^b_a,I)-w^+(k_a,I)=0$ always.

$\bullet$ If $a=3$, then the right hand side of (\ref{I'3}) is $1-\delta_{I'}$, hence
$w^b(k^b_a,I)-w^+(k_a,I)=0$ if and only if
$\max_{I'}\{\, \q^b(k^b_a+2E_{I'})\,\}$ can be realized by some $I'$ with $I'\ni j_0$. By (\ref{I'1})
this is equivalent to
\begin{equation}\label{eq:E2}
\mbox{$M(k)$ can be realized by some $I'$ with $I'\ni j_0$.}
\end{equation}

$\bullet$ If $a=-1$, then  the right hand side of (\ref{I'3}) is $\delta_{I'}$, hence
 $w^b(k^b_a,I)-w^+(k_a,I)=0$ if and only if
\begin{equation}\label{eq:E2b}
\mbox{$M(k)$ can be realized by some $I'$ with $I'\not\ni j_0$.}
\end{equation}


Assume that in the case $I\ni j_0$
for $a=3$ one has $w^b(k^b_a,I)-w^+(k_a,I)>0$. Then by an identical argument as in the case
$I\not\ni j_0$ we get that $\phi(A_q((k,I))-(\bar{k},I))=0$.

Hence, the remaining final case is when  there exists at least one $I'\subset I$ which contains $j_0$ and
 satisfies (\ref{eq:E2}). Then
 $$A_q((k,I)=(\bar{k},I)+(\bar{k}+2E^*_{j_0,\Gjp},I) \ \ \ (mod\ U \ \mbox{and} \ \ker(\phi)).$$
Now, we will consider $k+2E^*_{j_0,\Gamma}$ instead of $k$.  Via the operator $A_q$
it  provides the same cubes as $k$  (where the index set will have a shift $a\mapsto a-2$)
 but with different $U^m$--coefficients. Notice that by (\ref{eq:E2}) and (\ref{eq:E2b}),
 $M(k+2E^*_{j_0,\Gamma})$ can be realized only by subsets $I'$ with $I'\ni j_0$, hence  (\ref{eq:E2b}) will fail.
Therefore,
 $$A_q((k+2E^*_{j_0,\Gamma},I)=(\bar{k}+2E^*_{j_0,\Gjp},I)+(\bar{k}+4E^*_{j_0,\Gjp},I) \ \ \ (mod\ U).$$
More generally, for any positive integer $\ell$, by the same argument:
 $$A_q((k+2\ell E^*_{j_0,\Gamma},I)=(\bar{k}+2\ell
 E^*_{j_0,\Gjp},I)+(\bar{k}+(2\ell+2)E^*_{j_0,\Gjp},I) \ \ \ (mod\ U).$$
 Since $\phi$ is finitely supported, $\phi((2\ell+2)E^*_{j_0,\Gjp},I))=0$ for some $\ell\geq 0$, let us consider the minimal such $\ell$.
 Then $A_q$, modulo $U$ and $\ker(\phi)$, restricted on  the relevant finite dimensional spaces looks as an
 $(\ell+1)\times(\ell+1)$ upper triangular matrix whose diagonal is the identity matrix. Since this is invertible over $\Z$, the result follows: some linear combinations of the elements $A_q((k+2t E^*_{j_0,\Gamma},I)$, $0\leq t\leq \ell$, is
 $(\bar{k},I)$ modulo $U$ and $\ker(\phi)$.

\subsection{The surjectivity of $\mathbf{\overline{B}^q}$}\label{sur} We provide the same argument as \cite{JG}:
For any fixed $a\equiv e_{j_0} \ (mod\ 2)$,
$$\overline{B}^q\big(U^{-\ell}*(b(k)+aE^*_{j_0,\Gamma},I)^\veeK\big)=\pm U^{-\ell}*(k,I)^\veeK.$$
Since the collection of $U^{-\ell}*(k,I)^\veeK$ generate  $\calF^q(\Gj)$ (over $\Z$), the
surjectivity follows.

\subsection{$ \mathbf{\overline{B}^q\circ A^q=0}$}\label{ss} Take an
arbitrary $(k,I)\in \calQ_q(\Gj)$. Then, by (\ref{I'3}), one has:
$$(A_q\circ \overline{B}_q)((k,I))=\sum_{a\equiv 1}\ \sum_{\ \ c\equiv e_{j_0}}\  (-1)^{N(k)+\frac{c+e_{j_0}}{2}}\cdot
U^{\frac{a^2-1}{8}}\ \big(\pi_*b(k)+(a+c)E^*_{j_0,\Gjp},I\big).
$$
Those pairs $(a,c)$ for which $a^2-1$ and $a+c$ are fixed hit the same element of $\calF_q$. Write $a=2i+1$.
Then the two solutions
 of $(a^2-1)/8=i(i+1)/2=t$ satisfies $i_1+i_2=-1$. Since the corresponding two $c$ values
  satisfies $2i_1+c_1=2i_2+c_2$, one gets that $(c_2-c_1)/2$ is odd. Hence the terms cancel each other two
by two.

\subsection{$\mathbf{\ker \overline{B}^q\subset im\, A^q}$} Set $\ker U^m:=\{\phi\in\calF^q(\Gamma)\,:\, U^m\phi=0\}$.
Notice that the inclusion
\begin{equation}\label{ker}
\ker U^m\cap \ker \overline{B}^q\subset im\, A^q
\end{equation}
for $m=1$ (together with (\ref{ss}))
implies by induction its validity for any $m$ (cf. \cite{JG}).
Indeed, assume that the inclusion is true for $m-1$ and set $\phi\in \ker U^m\cap \ker \overline{B}^q$.
Then $U\phi=A^q(\psi)$ for some $\psi$. Moreover, $\tilde{\phi}:=\phi-A^q(U^{-1}*\psi)\in\ker U$. On the other hand, by (\ref{ss}), $\tilde{\phi}\in \ker \overline{B}^q$ too. Therefore, by (\ref{ker}) applied for $m=1$, we get
$\tilde{\phi}\in im \, A^q$, hence $\phi\in im\, A^q$ too.

Notice that $\cup_m(\ker U^m\cap \ker\overline{B}^q)=\ker \overline {B}^q$, hence this would prove
$\ker \overline{B}^q\subset im\, A^q$ too.

Next, we show (\ref{ker}) for $m=1$. First notice that
$\ker U\cap \ker \overline{B}^q$ is generated by elements of type
$(k,I)^\veeK+(k+2E^*_{j_0,\Gamma},I)^\veeK$ where $I\not\ni j_0$, and
$(k,I)^\veeK$ where $I\ni j_0$.

In the first case (i.e. $I\not\ni j_0$), let us write $k$ as in (\ref{k}), and set $\bar{k}$ as in (\ref{bark}).
Then by (\ref{I'3}) and (\ref{A}) one has
$$ A_q((k,I))=(\bar{k}-2E^*_{j_0,\Gjp},I)+(\bar{k},I) \ \ (mod\ U),$$
$$ A_q((k+2E^*_{j_0,\Gamma},I))=(\bar{k},I)+(\bar{k}+2E^*_{j_0,\Gjp},I) \ \ (mod\ U),$$
and $(\bar{k},I)$ does not appear in any other term with non--zero coefficient ($mod \ U$). Hence
$A^q(\bar{k},I)^\veeK=(k,I)^\veeK+(k+2E^*_{j_0,\Gamma},I)^\veeK$.

Next, we fix an element of type $(k,I)^\veeK$ with $I\ni j_0$. It belongs to the collection
$\{(k(i),I)\}_{a\in\Z}$, where $k(i)=k+2iE^*_{j_0,\Gamma}$, which will be treated simultaneously via the
discussions of (\ref{sINJ}). Write $k$ as in (\ref{k}) and set $\bar{k}$ via (\ref{bark});
it is also convenient to write $\bar{k}(i):=\bar{k}+2iE^*_{j_0,\Gjp}$.
Notice that $k(i)$ has coefficients $\{\{a_j\}_{j\not=j_0},a_{j_0}+2i\}$. Therefore, for $i\ll 0$
(\ref{eq:E2}) will fail  and (\ref{eq:E2b}) is satisfied.
Let $i_0-1$ be maximal when (\ref{eq:E2}) fails. Then for $i=i_0$ both conditions are satisfied, and for
$i>i_0$ only (\ref{eq:E2}). Therefore,
\begin{equation}
\begin{array}{lll}
A_q((k(i),I))=(\bar{k}(i-1),I)+(\bar{k}(i),I)& (mod \ U)& \mbox{if $i<i_0$,}\\
A_q((k(i_0),I))=(\bar{k}(i_0-1),I)+(\bar{k}(i_0),I)+(\bar{k}(i_0+1),I) & (mod \ U)& \\
A_q((k(i),I))=(\bar{k}(i),I)+(\bar{k}(i+1),I) &  (mod \ U)& \mbox{if $i>i_0$.}
\end{array}\end{equation}
This reads as
$$A^q((\bar{k}(i),I)^\veeK)=\left\{\begin{array}{ll}
(k(i),I)^\veeK+ (k(i+1),I)^\veeK & \mbox{if $i<i_0$}\\
(k(i),I)^\veeK                  & \mbox{if $i=i_0$}\\
(k(i),I)^\veeK+ (k(i-1),I)^\veeK & \mbox{if $i>i_0$}\end{array}\right . $$
Taking finite linear combination we get that any $(k(i),I)^\veeK$ is in the image of $A^q$.
This ends the proof of Theorem (\ref{Th1}). As a corollary we get:

\begin{theorem}\label{ES} Assume that the graphs $\Gjp, \ \Gamma$ and $\Gj$ are non--degenerate.
Then
$$\cdots \longrightarrow \bH^q(\Gjp)\stackrel{\bA^q}{\longrightarrow} \bH^q(\Gamma)\stackrel{\overline{\bB}^q}{\longrightarrow} \bH^q(\Gj)\stackrel{\bC^q}{\longrightarrow} \bH^{q+1}(\Gjp)\longrightarrow \cdots$$
is an exact sequence of $\Z[U]$--modules.
\end{theorem}

\section{The exact sequence for negative definite graphs}

\subsection{Preliminaries.}
For any graph  $\Gamma$ we write $\det(\Gamma)$ for the determinant of the {\it negative} of the  intersection form associated with $\Gamma$.
If $\Gamma$ is negative definite then $\det(\Gamma)$ is obviously positive.

If $\Gamma$ is negative definite then $\Gj$ is automatically so for any $j_0$.
Nevertheless, this is not the case for $\Gjp$ (although, if $\Gjp$ is negative definite then
$\Gamma$ is so too).

\begin{lemma}\label{l:int}
Assume that $\Gamma$ is negative definite. Then $\Gjp$ is negative definite if and only if
$\det(\Gamma)>\det(\Gj)$. If these conditions are satisfied then $E^*_{j_0,\Gamma}\not\in L(\Gamma)$
and $(E^*_{j_0,\Gamma})^2\not\in \Z$.
\end{lemma}
\begin{proof}
$\Gjp$ is negative definite if and only if $\det(\Gjp)$ is positive (provided that $\Gj$ is negative definite).
But $\det(\Gjp)=\det(\Gamma)-\det(\Gj)$. The last statement follows from
$-(E^*_{j_0,\Gamma})^2=\det(\Gj)/\det(\Gamma)$, which  cannot be an integer.
\end{proof}

In the sequel we assume that all the graphs are negative definite, but not
necessarily connected. The next theorem is an addendum of Theorem \ref{ES}.
Since its proof is rather long, it will be published elsewhere.

\begin{theorem}\label{ES2}
Assume that $\Gjp$ is negative definite. Consider the exact sequence
$$0\longrightarrow \bH^0(\Gjp)\stackrel{\bA^0}{\longrightarrow} \bH^0(\Gamma)\stackrel{\overline{\bB}^0}{\longrightarrow} \bH^0(\Gj)\stackrel{\bC^0}{\longrightarrow} \bH^{1}(\Gjp)\longrightarrow \cdots$$
and the canonical submodule  $\bt(\Gj)$ of \, $\bH^0(\Gj)$. Then
$\bt(\Gj)\subset im\, \overline{\bB}^0$; or, equivalently, the restriction $\bC^0|\bt(\Gj)$ is zero.
\end{theorem}

\subsection{Graphs with $n$ bad vertices.}\label{badver}
We say that a negative definite connected graph is `rational' if it is the plumbing representation
of a the link of a rational singularity (or, the resolution graph of a rational singularity). They were
characterized combinatorially by Artin, for more details see \cite{Five,NOSZ}.

We fix an integer $n\geq 0$. We say that a negative definite graph has at most $n$ `bad' vertices
if we can find $n$ vertices $\{j_k\}_{1\leq k\leq n}$, such that replacing their decorations $e_{j_k}$
by some more negative integers $e'_{j_k}\leq e_{j_0}$ we get a graph whose all connected components are rational.
(Notice that this is a generalization of the notion of `bad' vertices of \cite{OSzP}.
A graph with at most one bad vertex is called `almost rational' in \cite{NOSZ,NLC}. Any `star--shaped' graph,
i.e.  normal form of a Seifert manifold, has at most one bed vertex, namely the `central' vertex.)

\begin{theorem}\label{bad}
If \, $\Gamma$ has at most $n$ bad vertices then $\bH^q_{red}(\Gamma)=0$ for $q\geq n$.
\end{theorem}

This is a generalization of \cite[(4.3.3)]{NLC}, where it is proved for $n=1$ (compare also with the
vanishing theorems of \cite{OSzP,NOSZ}).

\begin{proof}
We run  induction over $n$. If $n=0$, then all the components of $\Gamma$ are rational. By \cite{NOSZ}, their
reduced lattice cohomology is vanishing. This fact remains true for more components too, since the cohomology
of a tensor product of two acyclic complexes is acyclic.

Assume now that the statement is true for $n-1$ and take
$\Gamma$ with $n$ bad vertices. Let $j$ be one of them. Let $\Gamma_{j}(-\ell)$ be the graph obtained from
 $\Gamma$ by replacing the decoration $e_{j}$ by $e_j-\ell$ ($\ell\geq 0$).  Then consider the long exact sequence (\ref{ES}) associated with $\Gamma_j(-\ell)$, $\Gamma_j(-\ell-1)$ and $\Gj$, for all $\ell\geq 0$.
 Then, by the inductive step, we get that $\bH^q(\Gamma)=\bH^q(\Gamma_j(-\ell))$ for all $\ell$ and $q\geq n$.
 (Here, in the case $n=1$, Theorem \ref{ES2} is also used.)
 Since for $\ell\gg 0$ the graph $\Gamma_j(-\ell)$ has only $n-1$ bad vertices, all these modules vanish.
\end{proof}

In fact, the above statement can be improved as follows.

\begin{theorem}\label{bad2}
Assume that  $\Gamma$ has at most $n\geq 2$ bad vertices  $\{j_k\}_{1\leq k\leq n}$
such that  $\Gamma\setminus j_1$ has at most $(n-2)$ bad vertices. Then $\bH^q_{red}(\Gamma)=0$ for $q\geq n-1$.
\end{theorem}

\begin{proof}
The proof is same as above, if one eliminates first the vertex $j_1$.
\end{proof}

See \cite[(8.2)(5.b)]{NOSZ} for a graph $\Gamma$ with 2 bad vertices $\{j_1,j_2\}$ such that $\Gamma\setminus j_1$
has only rational components.

\section{The `relative' surgery exact sequence}\label{s:REL}

\subsection{Preliminaries.} The motivation for the next exact sequence is two--folded. First, the exact sequence
(\ref{ES}) mixes the classes  of the characteristic elements. (Note that the surgery exact sequence
valid for the Heegaard--Floer theory --- which is one of out models for the theory ---  does the same.)
These classes, in topological language,  correspond to the $spin^c$--structures of the corresponding plumbed 3--manifolds. It would be desirable to have a surgery  exact sequence which do not mix them, and allows inductively the computation of each $\bH^*(\Gamma,[k])$ for each $[k]$ independently.

The second motivation is the main  result of \cite{BN}. This is a surgery formula for the Seiberg--Witten invariant of negative definite plumbed 3--manifolds, it compares these invariants for $\Gamma$ and $\Gj$ for fixed
(non--mixed) $spin^c$--structures. The third term in the main formula of \cite{BN} comes from a
`topological' Poincar\'e series
associated with the plumbing graph, and its nature is rather different than the other two terms.

Here our goal is to determine an exact sequence connecting $\bH^*(\Gamma,[\bar{k}])$ and $\bH^*(\Gj,[R(\bar{k})])$
(where $R(\bar{k})$ is the restriction of $\bar{k}$, see \ref{ss:notrem}) with the newly defined  third term,
playing the role of the relative cohomology.  Its relationship with
the Poincar\'e series used in \cite{BN} will also be treated.

\subsection{The `relative' complex and cohomology.}
We consider a non--degenerate graph $\Gamma$ and $j_0$ one of its vertices.

We fix $[k]\in Char(\Gj)/2L(\Gj)$ and a characteristic element
$k_m\in[k]$  with $w(k_m)=\min_{k\in[k]}w(k)$.
Furthermore, we fix $a_0$ satisfying $a_0\equiv e_{j_0}$ (mod 2). Then
 $k_{a_0}:=i(k_m)+ \big(\,(i(k_m),E_{j_0})+a_0\,\big)E^*_{j_0}\in Char(\Gamma)$.
 In fact, for any $k'\in[k]$ one gets that
$ i(k')+ \big(\,(i(k_m),E_{j_0})+a_0\,\big)E^*_{j_0}\in Char(\Gamma)$ and it is an element of
$[k_{a_0}]$. For simplicity we write $r_0$ for
$(i(k_m),E_{j_0})+a_0$.

We define $$B_{0,rel}:\calF_0(\Gamma\setminus j_0,[k])\to
\calF_0(\Gamma,[k_{a_0}])$$
by
\begin{equation}\label{eq:br2}
 B_{0,rel}(k')= i(k')+ \big(\,(i(k_m),E_{j_0})+a_0\,\big)E^*_{j_0}=i(k')+r_0E^*_{j_0}.
\end{equation}
This extends to the level of complexes $B_{*,rel}:(\calF_*(\Gamma\setminus j_0,[k]),\partial)\to
(\calF_*(\Gamma,[k_{a_0}]),\partial)$ by $B_{*,rel}((k,I))=(B_{0,rel}(k),I)$. Its dual
 $B^*_{rel}: (\calF^*(\Gamma,[k_{a_0}]),\delta)\to (\calF^*(\Gj,[k]),\delta)$
 is defined by $B^*_{rel}(\phi)=\phi\circ B_{*,rel}$.

By a similar argument as in (\ref{sur})  we get that
$$B^*_{rel}: \calF^*(\Gamma,[k_{a_0}])\to \calF^*(\Gj,[k]) \ \ \ \mbox{is surjective}. $$
We define the {\it `relative'} complex $\calF^*_{rel}=\calF^*_{rel}(\Gamma,j_0,[k],a_0])$ via
$\ker(B^*_{rel})$.
Let the cohomology of the complex $(\calF^*_{rel},\delta)$ be $\bH^*_{rel}=
\bH^*_{rel}(\Gamma,j_0,[k],a_0)$. It is a graded $\Z[U]$--module. 
We refer to it as the {\it relative lattice cohomology}.

Note that both $\calF^*_{rel}$ and $B^*_{rel}$ depend on the choice of the representative
$k_{a_0}$ of $[k_{a_0}]$ and are {\it not} invariants merely of the classes $[k_{a_0}]$ and $[k]$.

\begin{theorem}\label{ESk}
One has the short exact sequence of complexes:
$$0\longrightarrow \calF^*_{rel}\stackrel{A^*_{rel}}{\longrightarrow}
\calF^*(\Gamma,[k_{a_0}])\stackrel{B^*_{rel}}{\longrightarrow} \calF^*(\Gj,[k])
\longrightarrow 0, $$ which provides  a long exact sequence of \, $\Z[U]$--modules:
$$\cdots \longrightarrow \bH^q_{rel}\stackrel{\bA^q_{rel}}{\longrightarrow} \bH^q(\Gamma,[k_{a_0}])\stackrel{\bB^q_{rel}}{\longrightarrow} \bH^q(\Gj,[k])\stackrel{\bC^q_{rel}}{\longrightarrow} \bH^{q+1}_{rel}\longrightarrow \cdots$$
(Again, the relative cohomology modules and the operators in the above exact sequence depend on the
choice of $k_{a_0}$, and not only on its class $[k_{a_0}]$.)
\end{theorem}

\vspace{2mm}

In the case when $\Gamma$ is negative definite, the restriction of \,$\bB^0_{rel}$ to
$\et^+_{d[k_{a_0}]}$  has its image in $\et^+_{d[k]}$.

\begin{proposition}\label{prop:2}
Assume that $\Gamma$ (but not necessarily $\Gjp$)  is negative definite. Then
$$\bB^0_{rel}:\et^+_{d[k_{a_0}]}\to \et^+_{d[k]} \ \ \mbox{ is onto}.$$
 In particular,
 $\bH^*_{rel}$ has finite rank over $\Z$. Moreover, one has an exact sequence of finite $\Z$--modules:
 $$0\longrightarrow \bH^0_{rel}\longrightarrow \bH^0_{red}(\Gamma,[k_{a_0}])\oplus \Z^n
 \longrightarrow \bH^0_{red}(\Gj,[k])\longrightarrow \bH^{1}_{rel}\longrightarrow
 \bH^1_{red}(\Gamma,[k_{a_0}]) \longrightarrow \bH^1_{red}(\Gj,[k])
 \cdots$$
where $n:=w_\Gamma(i(k_m)+r_0E^*_{j_0})-\min \{\,w_\Gamma\mid[k_{a_0}]\}\in \Z_{\geq 0}$.
\end{proposition}
\begin{proof}
First note that
\begin{equation*}
w_\Gamma (i(k')+r_0E^*_{j_0})=w_{\Gamma\setminus j_0}(k') -\frac{1+r_0^2(E^*_{j_0})^2}{8}.
\end{equation*}
This applied for $k'=k_m$ provides
\begin{equation}\label{eq:ndk}
n +\frac{d[k_{a_0}]}{2}=\frac{d[k]}{2}-\frac{1+r_0^2(E^*_{j_0})^2}{8}.
\end{equation}
For any $l\geq 0$ and $\bar{k}\in [k_{a_0}]$
set $\bar{\phi}_l\in\calF^0(\Gamma,[k_{a_0}])$ defined by
\begin{equation*}\label{eq:phia}
\bar{\phi}_l(\bar{k})=
U^{-l-d[k_{a_0}]/2+\q_L(\bar{k})} \ \ \ (\bar{k}\in[k_{a_0}]).\end{equation*}
For different $l\geq 0$ they generate
$\et^+_{d[k_{a_0}]}\subset \bH^0(\Gamma,[k_{a_0}])$.
Similarly, set $\{\phi_l\}_{l\geq 0}$ in $\calF^0(\Gj,[k])$, where
\begin{equation*}\label{eq:phi}
\phi_l(k')=
U^{-l-d[k]/2+\q(k')} \ \ \ (k'\in[k]).
\end{equation*}
They generate $\et^+_{d[k]}$.
From the above identities and from the definition of $B^0_{rel}$ one gets for any $l\geq 0$:
$$\bB^0_{rel}(\bar{\phi}_{n+l})=\phi_l.$$
Hence the restriction $\bB^0_{rel}:\et^+_{d[k_{a_0}]}\to \et^+_{d[k]}$ is onto
and the $\Z$--rank of its kernel is $n$.
\end{proof}
The reader is invited to recall the definition of the
Euler characteristic of the lattice cohomology from (\ref{eq:EU}).
We define the {\it Euler characteristic of the relative lattice cohomology} by
\begin{equation*}
eu(\bH^*_{rel}):=\sum_{q\geq 0}(-1)^q \rank_\Z\, \bH^q_{rel}.
\end{equation*}
Then the exact sequence of Proposition~\ref{prop:2} and equation (\ref{eq:ndk}) provide
\begin{corollary}\label{cor:eu}  With the notation $r_0:=(i(k_m),E_{j_0})+a_0$ one has
$$eu(\bH^*_{rel}(\Gamma,j_0,[k],a_0))=eu(\bH^*(\Gamma,[k_{a_0}]))-eu(\bH^*(\Gamma\setminus j_0,[k]))
-\frac{1+r_0^2(E^*_{j_0})^2}{8}.$$
\end{corollary}
Fix any $l\in L(\Gamma\setminus j_0)$, then $k_{a_0}+2l=i(k_m+2l)+r_0 E^*_{j_0}$, hence we also get
\begin{corollary}\label{cor:eu2} For any $l\in L(\Gamma\setminus j_0)$ one has
\begin{equation}\label{eq:ueue}
\begin{split}
eu(\bH^*_{rel}(\Gamma,j_0,[k_m],a_0))=eu(\bH^*(\Gamma,[k_{a_0}]))
-\frac{(k_{a_0}+2l)^2_\Gamma+|\calj|}{8}\\
-eu(\bH^*(\Gamma\setminus j_0,[k_m]))
+\frac{(k_m+2l)^2_{\Gamma\setminus j_0}+|\calj\setminus j_0|}{8}.\end{split}\end{equation}
\end{corollary}
\subsection{Reinterpretation.}\label{ss:1re}
Above, we started with an element $k_m$ of the class $[k]$, and we constructed one of its extensions
$k_{a_0}$. Since $k_{a_0}+i(2l)=i(k_m+2l)+r_0E^*_{j_0}$, we have
$k_m+2l=R(k_{a_0}+2l)$ for any $l\in L(\Gj)$.

This procedure can be inverted. Indeed, let us fix any $\bar{k}\in Char(\Gamma)$  (which plays  the
role of $k_{a_0}+i(2l)$). Then set $R(\bar{k})$ and  define
$r_0$ by the identity $r_0E^*_{j_0}=\bar{k}-iR(\bar{k})$. Finally,
define $$B_{0,rel}:\calF_0(\Gamma\setminus j_0,[R(\bar{k})])\to
\calF_0(\Gamma,[\bar{k}]) \ \ \ \mbox{by} \ \ \
 B_{0,rel}(k')= i(k')+ r_0E^*_{j_0},$$
whose kernel is the relative complex $\calF^*_{rel}$ with cohomology $\bH^*_{rel}(\Gamma,j_0,\bar{k})$.
Then (\ref{eq:ueue}) reads as
\begin{equation}\label{eq:ueueue}
\begin{split}
eu(\bH^*_{rel}(\Gamma,j_0,\bar{k}))&=eu(\bH^*(\Gamma,[\bar{k}]))
-\frac{(\bar{k})^2_\Gamma+|\calj|}{8}\\ &
-eu(\bH^*(\Gamma\setminus j_0,[R(\bar{k})]))
+\frac{(R(\bar{k}))^2_{\Gamma\setminus j_0}+|\calj\setminus j_0|}{8}.
\end{split}\end{equation}
The identity (\ref{eq:ueueue}) depends essentially
on the choice of the choice of $\bar{k}\in Char(\Gamma)$. In fact, even if we fix the class
$[\bar{k}]$, the choice of the representative $\bar{k}$ from the class $[\bar{k}]$ provides essentially
different identities of type (\ref{eq:ueueue}): not only the terms
$ eu(\bH^*_{rel}(\Gamma,j_0,\bar{k}))$, $(\bar{k})^2_\Gamma$ and $(R(\bar{k}))^2_{\Gamma\setminus j_0}$
depend on the choice of $\bar{k}$, but even the class $[R(\bar{k})]$.

\subsection{The connection with the topological Poincar\'e series.}\label{ss:PoSe}
Let $K_\Gamma\in L'$ denote the  {\it canonical characteristic element } 
of $\Gamma$  defined by the  {\it adjunction formulae}
$(K_\Gamma+E_j,E_j)+2=0$ for all $j\in\calj$.
Similarly, one defines  $K_{\Gj}\in  Char(\Gj)$.  
Note that $Char(\Gamma)=K+2L'(\Gamma)$ and $K_{\Gj}=R(K_\Gamma)$.
The next result computes $eu(\bH^*_{rel}(\Gamma,j_0,K+2l'))$ 
in terms of $l'$ via the coefficients of a
series associated with $\Gamma$.

 Consider the  multi-variable
Taylor expansion $Z(\btt)=\sum p_{l'}\btt^{l'}$ at the  origin of
\begin{equation}\label{eq:INTR}\prod_{j\in \calj} (1-\btt^{E^*_j})^{\delta_j-2},\end{equation}
where for any $l'=\sum _jl_jE_j\in L'$ we write
$\btt^{l'}=\prod_jt_j^{l_j}$, and $\delta_j$ is the valency of $j$.
 This lives in $\Z[[L']]$, the submodule of
formal power series $\Z[[\btt^{\pm 1/d}]]$ in variables $\{t_j^{\pm 1/d}\}_j$, where $d=\det(\Gamma)$.
The series $Z(\btt)$ was used in several articles studying invariants of surface singularities, see
\cite{CDG,CDGEq,CHR,CDGb,coho3,NSW} for different aspects.

For any series $S(\btt)\in \Z[[L']]$, $S(\btt)=\sum_{l'}c_{l'}\btt^{l'}$, we have the
natural decomposition $$S=\sum_{h\in L'/L}S_h, \ \ \mbox{where} \ \  S_h:=\sum _{l'\,:\ [l']=h}\,
 c_{l'}\btt^{l'}.$$
In particular, for any fixed class $[l']\in L'/L$, one can consider the component
$Z_{[l']}(\btt)$ of $Z(\btt)$. In fact, see e.g. \cite[(3.1.20)]{CDGb},
\begin{equation}\label{eq:Po}
Z_{[l']}(\btt)=
\frac{1}{d}\sum _{\rho\in (L'/L)\,\widehat{}} \rho([l'])^{-1}\cdot
\prod_{j\in \calj} (1-\rho([E^*_j])\btt^{E^*_j})^{\delta_j-2},
\end{equation}
where $(L'/L)\,\widehat{}$ \ is the Pontjagin dual of $L'/L$.

Furthermore, once the vertex $j_0$ of $\Gamma$ is fixed,
for any class $[l']\in L'/L$ we  set
$$\cH_{[l'],j_0}(t):=Z_{[l']}(\btt)\big|_{t_{j_0}=t^d \ \ \ \ \ \ \ \atop \ t_j=1\ \mbox{\tiny{for}} \ j\not=j_0
}\ \in \Z[[t]].$$

 Let  $S(t) = \sum_{i\geq 0} c_i t^i$  be a formal power series.
  Suppose that for some positive integer $p$,
  the expression $\sum_{i=0}^{pn-1} c_i$ is a polynomial $P_p(n)$ in the
  variable $n$.  Then the constant term of $P_p(n)$ is independent of $p$.
  We call this constant term the \emph{periodic constant} of $S$ and
  denote it by $\mathrm{pc}(S)$ (cf.  \cite{NO1}).

\begin{proposition}\label{prop:PS}  Fix the vertex $j_0$ of $\Gamma$ and
write  $\cH_{[l'],j_0}(t)$ as $\sum_{i\geq 0}c_it^i$.

(a)  If $l'=\sum_ja_jE^*_j=\sum_jl'_jE_j\in L'(\Gamma)$ with all
$a_j$ sufficiently large
then $$\sum _{i <dl'_{j_0}}\, c_i = eu(\bH^*_{rel}(\Gamma,j_0,K+2l')).$$

(b) Take   $\bar{l}'=\sum_j\bar{l}'_jE_j\in L'(\Gamma)$ with $\bar{l}'_{j_0}\in [0,1)$.
Then $$\mathrm{pc}( \cH_{[\bar{l}'],j_0}) = eu(\bH^*_{rel}(\Gamma,j_0,K+2\bar{l}')).$$
\end{proposition}
\begin{proof}
Use (\ref{eq:ueueue}) from above  and the identities (3.2.7) and (3.2.13) from \cite{NSW}.
\end{proof}

\subsection{The connection with the Seiberg--Witten invariants.}\label{ss:SeWi}
Let $M(\Gamma)$ be the oriented plumbed 3--manifold associated with $\Gamma$, and $-M(\Gamma)$ the
same 3--manifold with opposite orientation. In is known that the $spin^c$--structures of
$M(\Gamma)$ (and of $-M(\Gamma)$ too) can be identified with $Char/2L$,
see e.g. \cite{NOSZ,NLC}. Let ${\bf sw}(M(\Gamma),[\bar{k}])$
be the Seiberg--Witten invariant of $M(\Gamma)$ associated with the $spin^c$--structure $[\bar{k}]$.
Then, by Theorem B of the Introduction of \cite{NSW}
for a negative definite graph  $\Gamma$  one has
\begin{equation}\label{eq:ThB}
eu(\bH^*(\Gamma, [\bar{k}]))={\bf sw}(-M(\Gamma),[\bar{k}]).
\end{equation}
Hence the above statements can be reinterpreted in terms of Seiberg--Witten invariants as well.

\begin{example}\label{ex:rat} Let $\Gamma$ be the next graph, where all the vertices have decoration
$-2$ except the $j_0$ vertex which has $-3$.

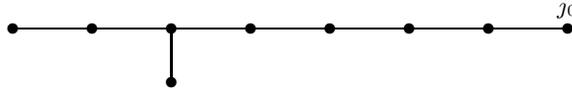
\begin{picture}(300,45)(-50,0)
\put(40,30){\circle*{4}} \put(70,30){\circle*{4}}
\put(100,30){\circle*{4}} \put(130,30){\circle*{4}}
\put(160,30){\circle*{4}} \put(190,30){\circle*{4}}
\put(220,30){\circle*{4}} \put(250,30){\circle*{4}}
\put(100,10){\circle*{4}}
\put(40,30){\line(1,0){210}} \put(100,30){\line(0,-1){20}}
\put(250,38){\makebox(0,0){$j_0$}}
\end{picture}

$\det(\Gamma)=\det(\Gj)=1$, hence for  both graphs we have only one class.

 Moreover, $\Gj$ is rational (an $E_8$--graph), hence $\min w_{\Gj}=0$, $\bH^*_{red}(\Gj)=0$ and
 $eu(\bH^*(\Gj))=0$.

 On the other hand, $\Gamma$ is minimally elliptic, $\min w_\Gamma=0$, $\bH^*_{red}(\Gamma)=
 \bH^0_{red}(\Gamma)= \Z_{(0)}$, the rank one $\Z$--module concentrated at  degree zero.
Hence $eu(\bH^*(\Gamma))=1$.

By the long exact sequence, we get $\bH^q_{rel}(\Gamma,j_0,\bar{k})=0$ for any $\bar{k}$ and $q>0$.

It is easy to see that $K_\Gamma=-E^*_{j_0}$, and $(E^*_{j_0})^2=-1$. Therefore, if $\bar{k}=K+2l'$,
and $r_0:=-(\bar{k},E^*_{j_0})$ and $l'_{j_0}:=-(l',E^*_{j_0})$, then
$r_0=2l'_{j_0}-1$.
By (\ref{cor:eu}) or (\ref{eq:ueueue})
$${\mathrm rank}\,\bH^0_{rel}(\Gamma,j_0,\bar{k})=eu(\bH^*(\Gamma,j_0,\bar{k}))=
\frac{r_0^2+7}{8}=1+\frac{l'_{j_0}(l'_{j_0}-1)}{2}.$$
By a computation  one obtains  $$\cH_{[0],j_0}(t)=\frac{1-t^6}{(1-t^3)(1-t^2)(1-t)}=
\frac{1-t+t^2}{(1-t)^2}=1+t+2t^2+3t^3+4t^4+\cdots.$$
Then $$\sum_{i<l'_{j_0}}c_i=1+1+2+3+\cdots +(l'_{j_0}-1)=1+\frac{l'_{j_0}(l'_{j_0}-1)}{2},$$
hence (\ref{prop:PS})(a) follows. In order to exemplify part (b), we have to take $\bar{l}'$
with $\bar{l}'_{j_0}=0$, hence in this case $eu(\bH^*(\Gamma,j_0,\bar{k}))=1$. But one also has
$${\mathrm pc}\ \frac{1-t+t^2}{(1-t)^2}=1.$$
\end{example}

\begin{example}\label{ex:end} Assume that $\Gamma$ is a star--shaped graph with central vertex
$j_0$ and we fix $\bar{k}=K_\Gamma$.

Since all the connected components of $\Gj$ are strings (i.e. rational graphs), for them  (cf. \cite{NLC})
$$eu(\bH^*(\Gj,[K_{\Gj}]))-\frac{(K_{\Gj})^2+|\calj\setminus j_0|}{8}=0.$$
Moreover, $\bH^q(\Gamma,[K])=\bH^q_{rel}(\Gamma,j_0,[K])=0$ for $q>0$, and
\begin{equation}\label{eq:pg}
\begin{split}
{\mathrm rank}\, \bH^0_{rel}(\Gamma,j_0,[K])=eu(\bH^*(\Gamma,[K])-\frac{K^2+|\calj|}{8}&=
 {\mathrm rank}\, \bH^0(\Gamma,[K])-\min w_\Gamma-\frac{K^2+|\calj|}{8}\\
 &= {\mathrm rank}\, \bH^0(\Gamma,[K])-\min \chi_K.\end{split}
\end{equation}
This equals the periodic constant of $\cH_{[0],j_0}(t)$ by \cite{NLC,BN}. Moreover, if $(X,o)$ is a
weighted homogeneous normal surface singularity with minimal good resolution graph $\Gamma$, then
its {\it geometric genus} $p_g$  equals the last term of (\ref{eq:pg}), cf. e.g. \cite{NN2,NOSZ}.
Hence $p_g(X,o)={\mathrm rank}\, \bH^0_{rel}(\Gamma,j_0,[K])$.
\end{example}

\end{document}